\documentclass[11pt,fleqn]{article}
\usepackage{amsmath,amssymb,amsthm,amsfonts}
\usepackage{geometry}
\newtheorem{theorem}{Theorem}
\newtheorem{lemma}{Lemma}
\newtheorem{definition}{Definition}
\newtheorem{proposition}{Proposition}

\newcommand{\Z}{\mathbb{Z}}
\newcommand{\R}{\mathbb{R}}
\newcommand{\T}{\mathbb{T}}

\newcommand{\norm}[1]{\left\|#1\right\|}
\newcommand{\abs}[1]{\left|#1\right|}
\newcommand{\floor}[1]{\left\lfloor#1\right\rfloor}
\newcommand{\dist}{\operatorname{dist}}

\title{\textbf{Proof of the Finiteness of the Chromatic Number of Two-Dimensional Lacunary Integer Distance Graphs}}
\author{\textbf{Nabh Singh}}
\date{May 5, 2026}

\begin{document}

\maketitle

\begin{abstract}
We extend the one-dimensional lonely set method to two dimensions for the purpose of studying the chromatic number of integer distance graphs in two dimensions. Given a lacunary sequence of displacement vectors in $\Z^2$, we use a lacunary matrix theorem given by Broderick, Fishman and Kleinbock, to prove the existence of a satisfactory multiplier vector. We then give an explicit geometric colouring argument. This proves that any integer distance graph generated by a lacunary sequence of vectors in two dimensions has finite chromatic number.
\end{abstract}

\section{Introduction}

Let
\[
D=\{\mathbf d_k\}_{k=1}^{\infty}\subset \Z^2\setminus\{\mathbf 0\}
\]
be a sequence of distinct displacement vectors. The integer distance graph $G(\Z^2,D)$ is defined as the graph with vertex set $V=\Z^2$, where an undirected edge exists between $\mathbf u,\mathbf v\in \Z^2$ if and only if
\[
\mathbf u-\mathbf v\in D
\quad\text{or}\quad
\mathbf v-\mathbf u\in D.
\]

In the one-dimensional setting, where $V=\Z$, it is a celebrated result of Katznelson \cite{katznelson} and Peres--Schlag \cite{peresschlag} that integer distance graphs associated with lacunary distance sets have finite chromatic number. In this note, we prove an analogous finiteness statement for displacement vectors in $\Z^2$ satisfying exponential growth in Euclidean norm.

\begin{definition}[2D lacunary sequence]
A sequence of vectors
\[
D=\{\mathbf d_k\}_{k=1}^{\infty}\subset \Z^2\setminus\{\mathbf 0\}
\]
is said to be lacunary if there exists a constant $r>1$ such that, for all $k\ge 1$,
\begin{equation}
    \norm{\mathbf d_{k+1}}_2 \ge r\norm{\mathbf d_k}_2,
\end{equation}
where $\norm{\cdot}_2$ denotes the Euclidean norm.
\end{definition}

\section{Lonely vector sets and torus-bin colourings}

Let
\[
\norm{x}_{\T}=\min_{m\in \Z}\abs{x-m}
\]
denote the distance from $x\in \R$ to the nearest integer.

\begin{definition}[Lonely vector set]
A distance set $D\subset \Z^2$ is called \emph{lonely} if there exist a multiplier vector
\[
\boldsymbol\alpha=(\alpha_1,\alpha_2)\in [0,1)^2
\]
and a constant $\delta>0$ such that
\begin{equation}
    \norm{\langle \mathbf d,\boldsymbol\alpha\rangle}_{\T}\ge \delta
    \qquad
    \forall \mathbf d\in D,
\end{equation}
where $\langle\cdot,\cdot\rangle$ denotes the standard inner product.
\end{definition}

We now show that loneliness implies finite colourability. The proof uses a two-dimensional torus-bin construction.

\begin{lemma}[Torus-bin colouring lemma]
Let $D\subset \Z^2\setminus\{\mathbf 0\}$ be a lonely vector set with gap $\delta>0$. Then
\begin{equation}
    \chi(G(\Z^2,D))
    \le
    \left(\floor{\frac{2}{\delta}}+1\right)^2.
\end{equation}
\end{lemma}

\begin{proof}
Let $\boldsymbol\alpha=(\alpha_1,\alpha_2)\in [0,1)^2$ satisfy
\[
\norm{\langle \mathbf d,\boldsymbol\alpha\rangle}_{\T}\ge \delta
\qquad
\forall \mathbf d\in D.
\]

Define the coordinate-wise torus map
\[
\Phi_{\boldsymbol\alpha}:\Z^2\to \T^2
\]
by
\begin{equation}
\Phi_{\boldsymbol\alpha}(\mathbf v)
=
(v_1\alpha_1 \bmod 1,\; v_2\alpha_2 \bmod 1).
\end{equation}

We use the $\ell_\infty$ metric on $\T^2$, defined by
\[
\rho_\infty(\mathbf x,\mathbf y)
=
\max_{i=1,2}\norm{x_i-y_i}_{\T}.
\]

Let $\mathbf u,\mathbf v\in \Z^2$ be adjacent vertices in $G(\Z^2,D)$. Then
\[
\mathbf u-\mathbf v=\pm \mathbf d
\]
for some $\mathbf d\in D$. Write
\[
a_i=(u_i-v_i)\alpha_i
\qquad
(i=1,2).
\]
Then
\[
\langle \mathbf u-\mathbf v,\boldsymbol\alpha\rangle
=
a_1+a_2.
\]

By the triangle inequality on $\T$,
\[
\norm{a_1+a_2}_{\T}
\le
\norm{a_1}_{\T}+\norm{a_2}_{\T}.
\]
Therefore
\[
\norm{a_1+a_2}_{\T}
\le
2\max\{\norm{a_1}_{\T},\norm{a_2}_{\T}\}.
\]
\begin{equation}
\implies \max\{\norm{a_1}_{\T},\norm{a_2}_{\T}\}
\ge
\frac12\norm{a_1+a_2}_{\T}.
\end{equation}

Since
\[
a_1+a_2
=
\langle \mathbf u-\mathbf v,\boldsymbol\alpha\rangle,
\]
we obtain
\begin{align}
\rho_\infty\!\left(
\Phi_{\boldsymbol\alpha}(\mathbf u),
\Phi_{\boldsymbol\alpha}(\mathbf v)
\right)
&=
\max_{i=1,2}\norm{(u_i-v_i)\alpha_i}_{\T} \\
&\ge
\frac12
\norm{\langle \mathbf u-\mathbf v,\boldsymbol\alpha\rangle}_{\T} \\
&=
\frac12
\norm{\langle \mathbf d,\boldsymbol\alpha\rangle}_{\T} \\
&\ge
\frac{\delta}{2}.
\end{align}

Thus adjacent vertices are separated by at least $\delta/2$ in the $\ell_\infty$ torus metric.

Now choose
\[
M=\floor{\frac{2}{\delta}}+1.
\]
Then,
\[
\frac1M<\frac{\delta}{2}.
\]
Partition $[0,1)^2$ into $M^2$ half-open square bins
\[
Q_{ij}
=
\left[\frac{i}{M},\frac{i+1}{M}\right)
\times
\left[\frac{j}{M},\frac{j+1}{M}\right),
\]
where $0\le i,j\le M-1$.

Assign one colour to each square bin. Namely, define
\[
c:\Z^2\to \{0,\dots,M-1\}^2
\]
by
\[
c(\mathbf v)=(i,j)
\quad\Longleftrightarrow\quad
\Phi_{\boldsymbol\alpha}(\mathbf v)\in Q_{ij}
\]

If two vertices $\mathbf u,\mathbf v$ receive the same colour, then their images lie in the same square bin. In other words,
\[
\rho_\infty\!\left(
\Phi_{\boldsymbol\alpha}(\mathbf u),
\Phi_{\boldsymbol\alpha}(\mathbf v)
\right)
<
\frac1M
<
\frac{\delta}{2}
\]
However, adjacent vertices satisfy
\[
\rho_\infty\!\left(
\Phi_{\boldsymbol\alpha}(\mathbf u),
\Phi_{\boldsymbol\alpha}(\mathbf v)
\right)
\ge
\frac{\delta}{2}
\]
which is a contradiction.
Thus $c$ is a proper colouring of $G(\Z^2,D)$ using $M^2$ colours, and
\[
\chi(G(\Z^2,D))
\le
M^2
=
\left(\floor{\frac{2}{\delta}}+1\right)^2
\]
\end{proof}

\section{Existence of a lonely multiplier}

We now use an existing theorem on lacunary matrix sequences to obtain the multiplier vector $\boldsymbol\alpha$.

\begin{theorem}[Broderick--Fishman--Kleinbock, lacunary matrix case]
Let $(M_k)_{k=1}^{\infty}$ be a lacunary sequence of real $m\times n$ matrices. Let $(Z_k)_{k=1}^{\infty}$ be a sequence of uniformly discrete subsets of $\R^m$. Then, the set
\[
\widetilde E(M,Z)
=
\left\{
\mathbf x\in \R^n:
\inf_{k\ge 1}\dist(M_k\mathbf x,Z_k)>0
\right\}
\]
is nonempty.
\end{theorem}

\begin{proposition}
Let
\[
D=\{\mathbf d_k\}_{k=1}^{\infty}\subset \Z^2\setminus\{\mathbf 0\}
\]
be lacunary. Then $D$ is lonely.
\end{proposition}

\begin{proof}
Write
\[
\mathbf d_k=(x_k,y_k).
\]
Regard $\mathbf d_k$ as a $1\times 2$ real matrix
\[
M_k=
\begin{pmatrix}
x_k & y_k
\end{pmatrix}.
\]
For $\boldsymbol\alpha=(\alpha_1,\alpha_2)\in \R^2$, we have
\[
M_k\boldsymbol\alpha
=
x_k\alpha_1+y_k\alpha_2
=
\langle \mathbf d_k,\boldsymbol\alpha\rangle.
\]

Moreover, the operator norm of $M_k$ is exactly
\[
\norm{M_k}_{\operatorname{op}}
=
\norm{\mathbf d_k}_2.
\]
Since $D$ is lacunary, there exists $r>1$ such that
\[
\norm{\mathbf d_{k+1}}_2
\ge
r\norm{\mathbf d_k}_2
\]
for every $k$. Hence
\[
\norm{M_{k+1}}_{\operatorname{op}}
\ge
r\norm{M_k}_{\operatorname{op}}.
\]
Thus $(M_k)$ is a lacunary sequence of $1\times 2$ real matrices.

Now take
\[
Z_k=\Z\subset \R
\]
for all $k$. The set $\Z$ is uniformly discrete. By the Broderick--Fishman--Kleinbock theorem, there exists
\[
\boldsymbol\alpha\in \R^2
\]
such that
\[
\inf_{k\ge 1}\dist(M_k\boldsymbol\alpha,\Z)>0.
\]
Therefore there exists $\delta>0$ such that
\[
\dist(M_k\boldsymbol\alpha,\Z)\ge \delta
\qquad
\forall k
\]
Equivalently,
\[
\norm{\langle \mathbf d_k,\boldsymbol\alpha\rangle}_{\T}\ge \delta
\qquad
\forall k
\]

Replace $\boldsymbol\alpha$ by its coordinate-wise reduction modulo $1$, which lies in $[0,1)^2$. Since each $\mathbf d_k$ has integer coordinates, this does not change
\[
\langle \mathbf d_k,\boldsymbol\alpha\rangle \pmod 1
\]
Thus there exist $\boldsymbol\alpha\in [0,1)^2$ and $\delta>0$ such that
\[
\norm{\langle \mathbf d_k,\boldsymbol\alpha\rangle}_{\T}\ge \delta
\qquad
\forall k
\]
That is, $D$ is lonely.
\end{proof}

\section{Main theorem}

\begin{theorem}
Let
\[
D=\{\mathbf d_k\}_{k=1}^{\infty}\subset \Z^2\setminus\{\mathbf 0\}
\]
be a lacunary sequence of vectors. Then the integer distance graph $G(\Z^2,D)$ has finite chromatic number.
\end{theorem}

\begin{proof}
By the above Proposition 1, the lacunary sequence $D$ is lonely. Hence there exist
\[
\boldsymbol\alpha\in [0,1)^2
\quad\text{and}\quad
\delta>0
\]
such that
\[
\norm{\langle \mathbf d,\boldsymbol\alpha\rangle}_{\T}\ge \delta
\qquad
\forall \mathbf d\in D.
\]
Using the torus-bin colouring lemma, we get
\[
\chi(G(\Z^2,D))
\le
\left(\floor{\frac{2}{\delta}}+1\right)^2
\]
Therefore $G(\Z^2,D)$ has finite chromatic number.
\end{proof}

\end{document}